\documentclass[draft]{amsart}
\usepackage{amsmath, amsfonts, amssymb, amsthm, amscd}

\theoremstyle{plain}
\newtheorem{theorem}{Theorem}[section]
\newtheorem{lemma}[theorem]{Lemma}
\newtheorem{corollary}[theorem]{Corollary}

\newtheorem{remark}[theorem]{Remark}
\newtheorem{claim}[theorem]{Claim}

\newtheorem*{ack}{Acknowledgement}

\theoremstyle{definition}

\newtheorem{example}{Example}[section]
\newtheorem{definition}[theorem]{Definition}

\input xy
\xyoption{all}

\def\p{{\frak p}}
\def\q{{\frak q}}

\def\m{{\frak m}}

\def\Coker{\operatorname{Coker}}
\def\div{\operatorname{div}}
\newcommand{\CH}{\operatorname{CH}}
\newcommand{\Spec}{\operatorname{Spec}}

\begin{document}
\subjclass{Primary 14C17, 14C20}
\title[Commutativity of Intersection with Divisors]
{An Algebraic Proof of the Commutativity of Intersection with Divisors}

\author{Paul Roberts}
\thanks{The first author was supported in part by NSF grant 0500588.}
\address{Department of Mathematics, University of Utah, Salt Lake City, UT 84112}
              
\email{roberts@math.utah.edu}

\author{Sandra Spiroff}
\address{Department of Mathematics, Seattle University, Seattle, WA  98122}
\email{spiroffs@seattleu.edu}


\begin{abstract} We present a purely algebraic proof of the commutativity of
the operation defined by intersection with divisors
on the Chow group of   a local Noetherian domain. 
\end{abstract}
\maketitle

\section*{Introduction} The operation given by intersecting with a
Cartier divisor is one of the basic ideas of Intersection
Theory, and the fact that it defines a commutative
operation in the Chow group is fundamental in making
the theory work.  If the intersection is proper, that is, if one intersects
a divisor $D$ with a variety $W$ not contained in $D$, this concept is quite simple.
However, if $W$ is contained in $D$, then even the basic definition is considerably
more complicated.  A classical approach to this question is to use
a ``moving lemma" to move $D$ to another divisor which
meets $W$ properly, while a newer method, introduced by  Fulton (\cite{F}),
is to use a theory of ``pseudo-divisors".   However, in the case of a local noetherian ring,
such an intersection must always be zero, and one can give a simple
definition in general.  
In spite of this, there has been no proof
of the crucial property that this operation is commutative in
the case of local rings  that did not involve the general definition, as well as
a considerable amount of machinery from Algebraic Geometry.  Proofs
of this fact   can be found in \cite{F} and \cite{R}; they use a
pullback to the blow-up of an ideal, the general theory for the resulting
divisors, and properties of proper morphisms of schemes.  

Our aim is to give an algebraic proof for this purely algebraic statement.
In the next section we give some background information as well as
precise definitions and a statement of our main result.  The following sections
reduce the problem to normal domains and give a more detailed statement of
the theorem. We then prove the theorem in a special case, and finally
give a proof of the general theorem, inducting on the
number of height one primes contained in the intersection.  (If the intersection
has codimension two, the proof of the result is easy).  

In our proof we give an explicit formula for the difference
between the intersections with two divisors taken in
different orders as a sum of divisors of rational functions
(see Theorem 3.1 in Section 3).  
This formula has been discovered previously in different contexts.
First, it has an interpretation in $K$-theory. 
Basically, the formula given in (6) below amounts to the assertion that the composition of the tame symbol and the $\div$ map in the Gersten complex is zero.  More specifically, for any Noetherian domain $R$ with field of fractions $K$, there is a complex \cite{Q}
$$K_2(K) \to \sum_{\text{ht}(\p)=1} \kappa(\p)^\times \to \sum_{\text{ht}(\q)=2} \mathbb Z,$$
and when $R$ is normal the first map is the tame symbol \cite{G}
$$\{\alpha, \beta \} \mapsto \sum_{\p} (-1)^{\nu_{\p}(\alpha) \nu_{\p}(\beta)} \cdot \frac{\alpha^{\nu_{\p}(\beta)}}{\beta^{\nu_{\p}(\alpha)}}$$ 
and the second map is $\div$.  That the Gersten complex is exact when $R$ is the localization of a finite type smooth $k$-algebra at a prime \cite{Q} leads to Bloch's formula (where $d$ = dim($R$))
$$H_{Zar}^2(X, \tilde{K_2}) \cong \CH_{d-2}(X).$$
The formula of Theorem 3.1 was also used by Kresch \cite{K} to give a more canonical geometric
proof of the commutativity that we prove here by algebraic means.

\section{Preliminaries}

We assume throughout that $A$ is a Noetherian ring.  In order to make intersection
theory work it is necessary to assume a few further properties that hold in
most situations that arise naturally.  First, we assume that there is a
good definition of dimension; that is, for all prime ideals $\p$ the dimension
of $A/\p$ is defined and that if $\p$ and $\q$ are distinct prime ideals such that
$\p\subset \q$ and there are no prime ideals strictly between $\p$ and $\q$, then
$\dim A/\p = \dim A/\q +1$.  The other condition we assume is that for all
$\p$, the normalization of $A/\p$ in its quotient field is a finitely
generated $A/\p$-module.  In particular, an excellent ring satisfies these properties.  For more details on these assumptions, we
refer to  \cite[Ch. 2]{F} and \cite[Ch. 8]{R}.

If $M$ is a module of finite length, we denote its length $\ell(M)$.
Let $Z_i(A)$ be the free abelian group with basis consisting of all prime
ideals $\q$ such that the dimension of $A/\q$ is $i$.  The elements of $Z_i(A)$ are called {\em cycles} of dimension $i$, and the basis element
corresponding to $A/\q$ is denoted $[A/\q]$.

\begin{definition} Let $\frak p$ be a prime ideal such that $\dim A/\p = i+1$ and $x$ an element of $A$ which is not in $\frak p$.  The cycle
$$ \sum \ell_{A_\q}(A_{\frak q}/(\frak p, x) A_{\frak q})[A/\q],$$
where the sum is taken over all $\frak q \in$ Spec($A$) such that dim$(A/\frak q)=i$,
is denoted by {\bf div($\frak p$, x)}, or occasionally $\div(A/\frak p, x)$.
\end{definition}

\begin{definition} {\bf Rational equivalence} is the equivalence
relation on $Z_i(A)$ generated by setting div$(\frak p,x)=0$ for all such primes $\frak p$ and elements $
x$.  We remark that if $x$ and $y$ are not in $\p$, then $\div(\p,xy)=\div(\p,x)+
\div(\p,y)$, and thus for any element $x/y$ in the fraction field of $A/\p$, we can
define $\div(\p,x/y)=\div(\p,x)-
\div(\p,y)$.
\end{definition}

\begin{definition} The $i$-th component of the Chow group of $A$, denoted by $\CH_i(A)$,
is $Z_i(A)$ modulo rational equivalence.  The {\bf Chow group} of $A$, denoted by $\CH_*(A)$, is obtained
by taking the direct sum of $\CH_i(A)$ for all $i$.  Similarly, the group of cycles
$Z_*(A)$ is the direct sum of the $Z_i(A)$.
\end{definition}

\begin{definition} \label{map} The intersection of a principal divisor $(u)$, where $u$ is an element in $A$, is a map $Z_*(A) \to Z_*(A/uA)$.  It is denoted by $(u) \cap -$ and referred to as {\bf intersection with (u)}.  On a basis element $[A/\frak p]$ it is defined by
$$
(u) \cap [A/\frak p] =\begin{cases}
0& \text{ if } u \in \frak p \\
 & \\
\displaystyle{\div(\p,u)}
   & \text{ if } u \notin\frak p
\end{cases}
$$
\end{definition}

If $\alpha = \sum n_i[A/\p_i]$ is an arbitrary cycle, it follows from the above definitions
that 
$$(u) \cap \alpha = \sum_{u\not\in \p_i} n_i \div(\p_i,u).$$
We note that if $u\not\in\frak p$,
then $(u) \cap [A/\frak p]$ is by definition rationally
equivalent to zero in the Chow group of $A$, but it is generally not rationally
equivalent to zero  in the Chow group
of $A/uA$.

Our main theorem is the following:

\begin{theorem}\label{commutativity}
Let $u$ and $v$ be elements of the ring $A$, and let $\alpha \in Z_i(A)$.  Then the cycles $(u)\cap (v) \cap \alpha$ and $(v)\cap(u) \cap \alpha$ are rationally equivalent in $Z_{i-2}(A/(u,v))$.
\end{theorem}


One of the main consequences of the theorem is that intersection with $(u)$ defines
an operation from the  Chow group of $A$ to the Chow group of $A/uA$. 

\begin{corollary}
The mapping on cycles that sends $\alpha$ to $(u) \cap \alpha$ induces a mapping
from $CH_*(A)$ to $CH_*(A/uA)$.  
\end{corollary}
\begin{proof}
 We must show that for any $\p\in \Spec(A)$ and any $x\not\in \p$, the
cycle $(u) \cap \div(\p,x)$ is rationally equivalent to zero as a cycle in $\Spec(A/uA)$.
By Theorem \ref{commutativity}, we have
$$(u) \cap \div(\p,x)=(u) \cap (x) \cap [A/\p] =(x) \cap (u) \cap [A/\p].$$
Let $(u) \cap [A/\p] = \sum n_i[A/\q_i]$.  Then each $\q_i$ contains $u$, so we may consider the 
$\q_i$ to be prime ideals in $A/uA$.  We thus have
$$(x) \cap (u) \cap [A/\p] = \sum_{x\not\in\q_i} n_i\div(\q_i, x),$$
which is clearly rationally equivalent to zero in the Chow group of $A/uA$.
\end{proof}

We remark that Theorem \ref{commutativity} is very easy to prove when the ideal generated
by $u$ and $v$ in $A/\p$ has height two and $\alpha = [A/\p]$; in this case the two cycles are in fact equal,
not just rationally equivalent.
To illustrate the general situation,  we give an example
where two elements intersect in codimension one.

\begin{example}  \label{ex} Let $A = k[x,y,z]$, where $k$ is a field.
We consider the intersections with the divisors defined by the elements
$xz$ and $xy$.  The following diagram shows the height one prime ideals that
contain these elements.

\vskip.25in

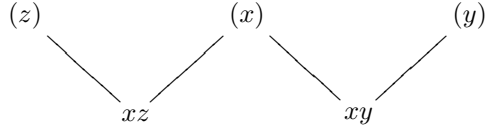
\xymatrix{
  & & & (z) \ar@{-}[dr] & & \frak (x)  \ar@{-}[dl]  \ar@{-}[dr]
   & & \frak (y) \ar@{-}[dl] \\
   & & & & xz & & xy }

By Definition \ref{map}, $$\displaystyle{(xz) \cap (xy) \cap [A] =
(xz) \cap \left([A/xA] + [A/yA]\right) = [A/(x,y)] + [A/(y,z)]},$$

and $$\displaystyle{(xy) \cap (xz) \cap [A] =
(xy) \cap \left([A/xA] + [A/zA]\right)= [A/(x,z)] + [A/(y,z)]}.$$
Clearly these cycles are not equal.  However, $(xz) \cap (xy) \cap [A] -  (xy) \cap (xz) \cap [A]$ $ =
\text{div} \left({(x)}, y/z \right),$ so they are rationally equivalent in $Z_1(A/(xy,xz))$.

\end{example}

In closing this section, we provide a statement
of the Approximation Theorem \cite[12.6]{M} since it is instrumental to our calculations.

{\bf Approximation Theorem:} {\it Let $K$ be the field of fractions of a
Krull domain $A$.  Given any set of height one primes  $\frak p_1, \dots, \frak p_r \in$ Spec($A$)
with corresponding discrete valuations $v_{\p_i}$,
and given integers $n_1, \dots, n_r$,  there is an element $x \in K^*$ such that
$v_{\frak p_i}(x) = n_i$ with $v_{\frak q}(x) \geq 0$ for all $\frak q \neq \frak p_i.$}

\section{Reduction to the case of a two-dimensional normal domain}

We first note that since we are proving a result for elements of the  group of cycles, we
can assume our element is a generator; that is, a cycle $[A/\frak p]$ for some prime
ideal $\frak p$.  Since the support of the cycles under consideration lie
in Spec$(A/\frak p)$ we can then assume that $\frak p = 0$ and we are dealing
with $[A]$ for an integral domain $A$.

As a first step in reducing to the case in which $A$ is a two-dimensional local
domain, we prove the following lemma.  

\begin{lemma}\label{chiformula}
Let $A$ be a one-dimensional local domain with maximal ideal $\m$, 
and let $x$ be a nonzero element of $A$.  For
a finitely generated $A$-module $M$, let $\chi(M) = \ell(M/xM)-\ell(_xM)$, where
$_xM = \{m\in M|xm = 0\}$.  Then
$$\chi(M)=\ell(A/xA)(rank(M)).$$
\end{lemma}

\begin{proof} The lengths involved are finite, and both sides of the equation  are additive
on short exact sequences.  Thus, by taking a filtration of $M$, we can
reduce to the cases where  $M =A$ or
$M = A/\frak m$.  For $M = A$ both sides are equal to
the length of $A/xA$, and for $M = A/\frak m$ both sides are zero.
\end{proof}

\medskip

Now let $A$ and $B$ be integral domains,  let $B$ be a finite
extension of $A$, and let $\Phi$ be the induced map from $\Spec(B)$
to $\Spec(A)$.   We define a map $\Phi_*$ from cycles on $B$ to cycles on $A$
by letting
$$\Phi_*([B/\frak P]) = [\kappa(\frak P): \kappa(\frak p)][A/\frak p],$$
where $\frak p = A \cap \frak P$.  Here $[\kappa(\frak P): \kappa(\frak p)]$ denotes the degree of the extension of residue fields, which is finite since $B$ is a finite extension of $A$.  The
next lemma is a special case of the projection formula in intersection theory.

\begin{lemma}\label{projection}
Let $A\subset B$ be as above, and let $u$ be a nonzero element of $A$ (and
thus also of $B$).  Then for any cycle $\eta$ on $B$, the cycles $\Phi_*((u) \cap \eta)$
and $(u) \cap (\Phi_*(\eta))$ are equal.
\end{lemma}

\begin{proof}  It suffices to prove the result for a cycle of the form $[B/\frak P]$,
and in addition we may assume that $\frak P=0$.  (If $u\in \frak P$, then both cycles
are zero.)  The cycle $\Phi_*([B])$ is $r[A]$, where $r$ is the rank of $B$
as an $A$-module. Thus if $\frak q$ is a height one prime of $A$ containing $u$, the
coefficient of $[A/\frak q]$ in  $(u) \cap (\Phi_*([B]))$ is $\ell(A_{\frak q}/u A_{\frak q})$
times $r$, and $r$ is also the rank of $B_{\frak q}$ over $A_{\frak q}$.  By Lemma 
\ref{chiformula}, this is equal to
the length of $B_{\frak q}/uB_{\frak q}$ as an $A_{\frak q}$ module
(since in this case there are no nonzero elements
annihilated by $u$).  By taking a filtration of $B_{\frak q}/uB_{\frak q}$ with quotients
$B_{\frak q}/\frak Q B_{\frak q}$ for primes $\frak Q$ containing
$u$, we get
 $$\ell_{A_{\q}} (B_{\frak q}/uB_{\frak q})=\sum_{\frak Q}[\kappa(\frak Q): \kappa(\frak q)]\ell(B_{\frak Q}/
 uB_{\frak Q}).$$
The right hand side of this equation is the coefficient of $[A/\q]$ in $\Phi_*((u) \cap [B])$,
so this proves the lemma. 
\end{proof}
\medskip

We also need the following result, which is a special case of ``proper push-forward" of cycles.  If
the field $L$ is a finite extension of a field $K$, we denote the norm from $L$ to $K$
by $N_{L/K}$; recall that $N_{L/K}(x)$ is the determinant of the map given by
multiplication by $x$ on $L$ considered as a vector space over $K$.

\begin{lemma}\label{properpf}

Let $A$ be a local one-dimensional domain.
\begin{enumerate}  
\item  Let $M$ be a finitely generated
torsion-free $A$-module, and let $\phi$ be  an $A$-endomorphism of $M$
such that $Coker(\phi)$ has finite length.  Let $K$ be the quotient field of $A$, and
let $k=a/b$ be the determinant of the induced endomorphism on $M\otimes K$, where $a$ and
$b$ are in $A$. Then
$$\ell(Coker(\phi))=\ell(A/aA)-\ell(A/bA).$$

\item Let $B$ be an integral domain containing $A$ that is a finitely generated $A$-module, and set $L$ and $K$ to be their quotient fields, respectively.  Let $k$ be an element of $L$, and let $\Phi_*$ be defined as above.  Then 
$$\Phi_*(\div(B,k))=\div(A,N_{L/K}(k)).$$
\end{enumerate}
\end{lemma}
\begin{proof}  To prove (1), let $\overline A$ be the integral closure of $A$ in $K$,
which we are assuming is a finitely generated $A$-module, and let $\overline M$
be the $\overline A$-module generated by $M$ in $M\otimes_A K$.   
Then $\phi$ extends to an endomorphism
of $\overline M$ and thus also to an endomorphism of $\overline M/M$,
which has finite length. An application of the Snake Lemma
shows that the length of the cokernel of $\phi$ on $M$ is equal to
the length of the cokernel of its extension to $\overline M$
(we note that since $\Coker(\phi)$ has
finite length and $M$ is torsion-free, $\phi$ is injective).  Similarly,
the lengths of $A/aA$ and $A/bA$ are equal to the lengths of
$\overline A/a\overline A$ and $\overline A/b\overline A$.  Thus we may
assume that $A$ is integrally closed in its quotient field so is a semi-local Dedekind domain.  In this case $A$ is a principal ideal domain, so we can
diagonalize $\phi$ and the result is clear.

It suffices to prove (2) for $k = b\in B$, and from part (1) it
suffices to show that for $\p \in \Spec(A)$ of height one, the length of $B_{\p}/bB_{\p}$ is
equal to $$\sum_{\frak P} [\kappa(\frak P): \kappa(\p)]\ell_{B_{\frak P}}(B_{\frak P}/bB_{\frak P}),$$
where the sum is taken over all $\frak P$ lying over $\p$.  This formula
follows immediately from taking a filtration of $B_{\p}/bB_{\p}$ with
quotients of the form $B_{\frak P}/\frak P B_{\frak P}$.
\end{proof}

\begin{theorem}\label{redtonormal}  (Reduction to the normal case).
Let $u,v$ be elements of an integral domain $A$ of dimension $d$, and let $B$ be the normalization of $A$ in its quotient field.  If $(u) \cap (v) \cap [B]$ and $(v) \cap (u) \cap [B]$ are rationally equivalent in
$Z_{d-2}(B/(u,v))$, then $(u) \cap (v) \cap [A]$ and $(v) \cap (u) \cap [A]$ are rationally equivalent in
$Z_{d-2}(A/(u,v))$.
\end{theorem}
\begin{proof}  Let $\frak P_i$ be the height one prime ideals of $B$ in the support of $(u,v)$,
and let $\p_i$ be their intersections with $A$; we note that the $\p_i$ are exactly the height
one primes of $A$ that contain $(u,v)$.  Let $k_i$ be rational functions on
$B/\frak P_i$ such that we have an equality of cycles
$$(u) \cap (v) \cap [B]-(v) \cap (u) \cap [B] = \sum \div(\frak P_i,k_i).\eqno{(5)}$$
Now from Lemma \ref{projection}, we have
$$\Phi_*((u) \cap (v) \cap [B]) = (u) \cap \Phi_*((v) \cap [B])  = (u) \cap (v) \cap \Phi_*([B]),$$
and similarly 
$$\Phi_*((v) \cap (u) \cap [B]) = (v) \cap (u) \cap \Phi_*([B]) .$$
Since $\Phi_*([B]) = [A]$ (as $B$ is finitely-generated over $A$), applying $\Phi_*$ to the left hand side of equation (5) gives
$$(u)\cap (v) \cap [A] - (v)\cap (u) \cap [A].$$
On the other hand, if we apply $\Phi_*$ to the right hand side,
by Lemma \ref{properpf} we obtain 
$$\sum_{\frak P_i} \div(\p_i,N_{\kappa(\frak P_i)/\kappa(\p_i)}(k_i)).$$
This shows that $(u)\cap (v) \cap [A] - (v)\cap (u) \cap [A]$ is rationally equivalent to zero in
$Z_{d-2}(A/(u,v))$. 
\end{proof}

\bigskip
 
In summary, we may assume that $A$ is a normal domain and that the cycle we are
intersecting is $[A]$.  The reduction to dimension two occurs in the next section.

\section{A formula for the cycle $(u) \cap (v) \cap [A] - (v) \cap (u) \cap [A]$}

We begin by setting up the general situation we will be considering and then give a
formula for the difference of the cycles in terms of elements of the form
$\div(\frak p_i, k_i)$, for rational functions $k_i$ on $A/\p_i$.  The remainder of the paper is devoted to proving the formula.

Our situation is depicted below:

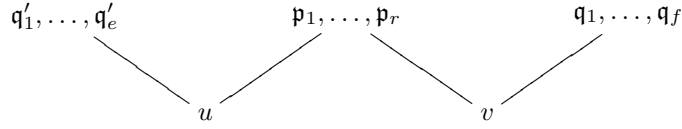
\xymatrix{
 & & \frak q'_1, \dots, \frak q'_e \ar@{-}[dr] & & \frak p_1, \dots, \frak p_r \ar@{-}[dl] \ar@{-}[dr]
  & & \frak q_1, \dots, \frak q_f \ar@{-}[dl] \\
   & & & u & & v }

All of the prime ideals shown are height one primes of $A$, and 
the $\frak q_k'$, $\frak p_i,$ and $\frak q_l$ are
those primes that contain only $u$, both $u$ and $v$, and only $v$, respectively.  Since $A$ is a normal
domain, the localization at every height one prime $\frak p$ is a discrete valuation ring and
defines a valuation $\nu_{\frak p}$. We let the orders of $u$ and $v$ at the primes displayed above be as follows:
$$\nu_{\frak q'_k}(u)=s_k \hskip.25in \nu_{\frak p_i}(u)=n_i  \hskip.25in \nu_{\frak p_i}(v)=m_i \hskip.25in
\nu_{\frak q_l}(v)=t_l$$
If $A$ has dimension $d$, the prime ideals $\frak p$ with $\dim(A/\frak p) = d-1$ in $A/(u,v)$ are the
images of the $\frak p_i$.  Hence to show that the cycle is rationally equivalent to zero,
we must show that it is a sum of cycles of the form $\div(\frak p_i,k_i)$.  Our main theorem (a more detailed statement of Theorem \ref{commutativity}) gives the rational functions that make this work.

\begin{theorem} \label{formula} Let $\frak p_i$ be the height one prime ideals of a Noetherian normal domain $A$ containing $u, v \in A$ as above.
If, for each $i$ between 1 and $r$, the pair of elements $a_i, b_i \in A$ is not in $\frak p_i$ and satisfies
$$\frac{a_i}{b_i} = \frac{v^{n_i}}{u^{m_i}},$$ then there is an equality of cycles
$$(u)\cap (v) \cap [A]- (v)\cap (u) \cap [A] = \sum_{i=1}^r \div(\frak p_i, a_i/b_i).\eqno{(6)}$$
\end{theorem}

By the Approximation Theorem, there always exists elements $a_i$ and $b_i$ in $A \backslash \frak p_i$ such that $a_i/b_i = v^{n_i}/u^{m_i}$.  In the course of the proof we will give a particular choice of $a_i$ and $b_i$, but we note that the cycle
on the right is independent of the choice as long as the elements satisfy the hypotheses of the theorem; i.e., $\div(\frak p_i, a_i/b_i) = \div(\frak p_i, c_i/d_i)$ whenever $a_i/b_i = c_i/d_i$ and $a_i, b_i, c_i, d_i \notin \frak p_i$.

We also note that since this is an equality of cycles, it is enough to check that the coefficient
of $[A/\frak m]$ is the same for both sides of the equation for every prime
ideal $\frak m$ of height two.  Thus, by localizing we may
assume that $A$ is a local normal domain of dimension two and that $\frak m$ is its maximal ideal.

In summary then, to establish (6) we show the following equality:
\vskip.25in

\xymatrix{
 & *+[F]{\displaystyle{\sum_{l=1}^f t_l\ell(A/(\q_l,u))[A/\frak m]-\sum_{k=1}^e s_k\ell(A/(\q'_k,v)) [A/\frak m]= 
\sum_{i=1}^r  \div(\p_i,a_i/b_i)},}}
\vskip.25in

where we will often omit writing the basis element $[A/\frak m]$ on the left hand side and use $\div(\p_i, a_i/b_i)$ to denote the coefficient of $[A/\frak m]$ on the right hand side.

\section{First Step in the Induction Argument}

From this point on, we assume that $A$ is a local normal domain 
of dimension two and that the elements $u, v$
of $A$
intersect in codimension one.  We remark that in the case where $u$ and $v$ generate a height two ideal, since we are assuming that $A$ is a normal
domain, $u,v$ form a regular sequence and hence both $(u) \cap (v) \cap [A]$ and $(v) \cap (u) \cap [A]$ give the length of $A/(u,v)$.  As a result, the right hand side of equation (6) is zero.  In the case we are considering, where
$u$ and $v$ generate a height one ideal, $A/(u,v)$ no longer has finite length, but this
quotient, or more precisely a subquotient, is still the starting point for the computation.

In this section we prove the special case where
 $m_i = n_i$ for each $i$, which, as we show below, implies the case of a single prime.   We recall that $m_i=\nu_{\frak p_i}(v)$ and $n_i=\nu_{\frak p_i}(u)$, so the 
assumption says that $u$ and $v$ have the same order for each $\frak p_i$. As a result, only one pair of elements $a,b \in A$ is necessary.  In the next section we will
 prove the general case by using this one.

\begin{theorem} \label{mequalsn} Let $\frak p_i$ and $u, v$ be as in \ref{formula}.  If $n_i=m_i$ for all $i$, and $a$ and $b$ are elements of $A$ not in any of the $\frak p_i$ such that $a/b = v/u$, then we have an equality of cycles
$$(u)\cap (v) \cap [A]- (v)\cap (u) \cap [A] = \sum_{i=1}^r \div \left(\frak p_i, \frac{a^{n_i}}{b^{n_i}} \right).$$
\end{theorem}

\begin{proof}  Let $P=\cap_{i=1}^r\p_i^{(n_i)}$.  Then $u$ and $v$ are in the ideal $P$ and, since $\nu_{\p}(u)=\nu_{\p}(P)$ or $\nu_{\p}(v)=\nu_{\p}(P)$
for all height one prime ideals $\p$ of $A$, $P/(u,v)A$ is a module of finite length.  Our proof consists of expressing
the length of this module in different ways.

Let $Q=\cap_{l=1}^f\q_l^{(t_l)}$ and $Q'=\cap_{k=1}^e{\q'_k}^{(s_k)}$.  
We note that $Q\cap P=vA$ and $Q'\cap P=uA$.

We claim that we have a short exact sequence
$$0\to A/(Q + P)\stackrel{u}{\to}P/(vA+uP)\to P/(u,v)A \to 0.$$
To see this, we note that if $a\in Q$, then $ua\in vA$, so multiplication by $u$ does take
$Q + P$ to $vA + uP$.  Conversely, if $ua = va' + up$, for $p \in P$, then $u(a - p) \in vA$.  This happens
exactly when $a - p \in Q$, which implies that $a\in Q + P$. It is clear that the image
of this map is $(u,v)A/(vA + uP)$, so exactness at the other places holds.

Interchanging $u$ and $v$ yields a similar short exact sequence.  Combining these, we deduce that
$$\ell(P/(vA+uP))-\ell(A/(Q+P))= \ell(P/(uA+vP))-\ell(A/(Q' + P)).$$
Consider the term $\ell(P/(vA+uP))$.  The height one prime ideals in the support of
$P/vA$ are the $\q_l$.  Since $u$ is not contained in any of these, we determine that 
multiplication by $u$ on $P/vA$ is injective; its cokernel is  $P/(vA+uP)$.
Furthermore, since $P/vA$ has a filtration with quotients $A/\q_l$ of multiplicity
$t_l$, we obtain
$$\ell(P/(vA+uP))=\sum_{l=1}^f t_l\ell(A/(\q_l,u)).$$
Similarly, we have
$$\ell(P/(uA+vP))=\sum_{k=1}^e s_k\ell(A/(\q'_k,v)).$$
Combining these terms, we obtain
$$\sum_{l=1}^f t_l\ell(A/(\q_l,u))-\sum_{k=1}^e s_k\ell(A/(\q'_k,v))=\ell(P/(vA+uP))-\ell(P/(uA+vP)),$$
and from the previous equation this difference is equal to
$$\ell(A/(Q+P))-\ell(A/(Q'+P)).$$
It now remains to prove that if we have $a$ and $b$ not in $\p_i$ for any $i$ with
$a/b=v/u$, then
$$\ell(A/(Q+P))-\ell(A/(Q'+P))=\sum_{i=1}^r  \div \left(\p_i,\frac{a^{n_i}}{b^{n_i}} \right).$$
From the Approximation Theorem, we can find an element $a \in A$ such
that $a$ avoids all the $\frak p_i$ and $\frak q'_k$, but $\nu_{\frak q_l}(a)=\nu_{\frak q_l}(v) = t_l$ for 
all $l$.  Additionally, we might have $a \in J_h$ for a finite collection of height one primes
  $J_h$.  Let $\lambda_h$ be the order of $a$ in $J_h$.  Set $b = ua/v$.  Then $b \in A$.  In particular, $b$ avoids every $\frak p_i$ and $\frak q_l$, $\nu_{\frak q_k'}(b)= \nu_{\frak q_k'}(u) = s_k$ for every $k$, and $\nu_{J_h}(b)= \lambda_h$ for all $h.$

Let $K$ be the quotient field of $A$.  We next consider the composition of multiplications, $$P^{(-1)}/A\stackrel{u}{\to}A/P\stackrel{a}{\to}A/P,$$

where $P^{(-1)} = \cap_{i=1}^r \frak p_i^{(-n_i)}$ and $\frak p_i^{(-n_i)} = \{x \in K: \nu_{\p_i}(x) \geq -n_i \}$.

The kernel-cokernel exact sequence gives us a short exact sequence
$$0\to \Coker(u)\to \Coker(ua)\to \Coker(a)\to 0.$$
The first cokernel is $A/(P+uP^{(-1)}) = A/(P+Q').$  The length of the cokernel of multiplication by $a$ on 
$A/P$ is, by looking at a filtration of $A/P$ with quotients of the form $A/\p_i$,
$$\sum_{i=1}^r n_i \ell(A/(\p_i, a)) = \sum_{i=1}^r n_i \div(\p_i,a).$$
Thus the above short exact sequence gives 
$$\ell(\Coker(ua))=\ell(A/(P+Q'))+\sum_{i=1}^r n_i \div(\p_i,a)=\ell(A/(P+Q'))+\sum_{i=1}^r \div(\p_i,a^{n_i}).$$
A similar computation gives
$$\ell(\Coker(vb))=\ell(A/(P+Q))+\sum_{i=1}^r n_i \div(\p_i,b)=\ell(A/(P+Q))+\sum_{i=1}^r \div(\p_i,b^{n_i}).$$
Since $ua=vb$, we obtain
$$\ell(A/(P+Q))-\ell(A/(P+Q')) = \sum_{i=1}^r \div(\p_i,a^{n_i})-\sum_{i=1}^r \div(\p_i,b^{n_i})
=\sum_{i=1}^r \div(\p_i, \frac{a^{n_i}}{b^{n_i}}),$$
which proves the theorem.
\end{proof}

As mentioned, the above argument implies the case in which there is only one height one prime $\p$ over $(u,v)$.  This will establish the first step in the induction argument.

\begin{corollary} \label{oneprime} With the same hypotheses of \ref{formula} and $r=1$, we have an equality of cycles
$$(u)\cap (v) \cap [A]- (v)\cap (u) \cap [A] = \div \left(\frak p, \frac{a}{b} \right).$$
\end{corollary}

We apply the previous argument to $u^m$ and $v^n$, where $\nu_\p(u)=n$ and $\nu_\p(v)=m$.  In this case, $P = \frak p^{(mn)}$, $Q = \cap_{l=1}^f \frak q_l^{(mt_l)}$, $Q' = \cap_{k=1}^e {\frak q'_k}^{(ns_k)}$, and $a/b = v^n/u^m.$  The resulting equality of cycles is $$(u^m)\cap (v^n) \cap [A]- (v^n)\cap (u^m) \cap [A] = \div \left(\frak p, \frac{a^{mn}}{b^{mn}} \right),$$
which simplifies to the one shown.

There is another important application of Theorem \ref{mequalsn}.  With the notation as above, the roles of the pairs $\{u,v\}$ and $\{a,b\}$ can be interchanged.  Of course, as a result the ideals $\frak p_i$ and $J_h$ must also swap roles.  

\begin{corollary}  \label{ab} With the same notation as in the proof of \ref{mequalsn}, we have
$$(b)\cap (a) \cap [A]- (a) \cap (b) \cap [A] = \sum_h \div \left(J_h, \frac{v^{\lambda_h}}{u^{\lambda_h}} \right).$$
\end{corollary}

\section{The General Induction Argument}

We are now in a position to prove the general result.  In the previous section we established this result in two cases, and the condition that made these proofs possible was that the ratios
$n_i/m_i$ were the same for all $i$, or in the case of Corollary \ref{oneprime}, that there was
only one $i$.  In the general case this will not hold.  The general proof is
by induction on the number of primes of height one containing $(u,v)$.  Since the ratios 
$n_i/m_i$ and $n_j/m_j$ are not necessarily the same for different $i$ and $j$, the numbers $n_im_j - m_in_j$ will not all be zero, and this will effect our choice of $a_i$ and $b_i$.

We now prove our theorem in general.

\begin{proof}
Assume that $r \geq 2$ and that the result holds when there are $r-1$ primes.  Specifically, our induction hypothesis is: {\it Given a pair of elements $x$ and $y$ in 
$A$ that intersect in some} proper {\it subset $\mathcal S$ of $\frak p_1, \dots, \frak p_r$, there is an equality of cycles $$(x)\cap (y) \cap [A]- (y)\cap (x) \cap [A] = \sum_{\frak p_i \in \mathcal S} \div \left(\frak p_i, \frac{c_i}{d_i} \right),$$
for elements $c_i, d_i$ not in $\frak p_i$ such that $c_i/d_i = y^{\nu_{\frak p_i}(x)}/x^{\nu_{\frak p_i}(y)}$.}

As in Corollary \ref{oneprime}, we want to use the Approximation Theorem to find elements $a_1, \dots, a_r$ and $b_1, \dots, b_r$ such that for each $i$, 
$$\displaystyle{\frac{a_i}{b_i} = \frac{v^{n_i}}{u^{m_i}}}.$$  
After a possible reordering of the primes $\frak p_i$, we may assume that 
$$n_1/m_1 \geq n_2/m_2 \geq \cdots \geq n_r/m_r.$$
Let $G$ be the integer such that 
$n_1/m_1 = n_2/m_2 = \cdots = n_G/m_G > n_{G+1}/m_{G+1}$; then $1\le G < r$.
For $j \geq 1$, set $\alpha_j = n_1m_j - m_1n_j$. We have $\alpha_1 = \cdots = \alpha_G = 0$,
and $\alpha_j>0$ for $j\ge G+1$.

Using the Approximation Theorem, choose $a_1$ such that 
$$\div(a_1) = \sum_{j=G+1}^r \alpha_j [A/\frak p_j] + \sum_{l=1}^f n_1t_l [A/\frak q_l] + \sum_h \lambda_h [A/J_h],$$
where the $J_h$ are a finite number of height one primes of $A$ and $\lambda_h > 0$.  Set 
$\displaystyle{b_1 = \frac{u^{m_1}a_1}{v^{n_1}}}$.  Then  $\div(b_1) =  \sum_{k=1}^e m_1s_k [A/\frak q_k'] +
\sum_h \lambda_h [A/J_h]$.  It is important to note that $a_1$ and $b_1$ do not intersect on any of the 
original primes $\frak p_j, \frak q_l, \frak q_k' $; they only intersect on the primes $J_h$ and their orders are equal for each $h$.  This is exactly the scenario of Corollary \ref{ab}, using the relation $v^{n_1}/u^{m_1} = a_1/b_1$.  The explicit formula from Corollary \ref{ab} is shown below.
$$(b_1) \cap (a_1) \cap [A] - (a_1) \cap (b_1) \cap [A] =  \sum_h \div \left(J_h, \frac{v^{n_1 \lambda_h}}{u^{m_1\lambda_h}} \right)\eqno{(7)}$$
Moreover, a direct calculation shows it is also true that $$(v^{n_1}) \cap (a_1) \cap [A] - (u^{m_1}) \cap (b_1) \cap [A] =  \sum_h \div \left(J_h, \frac{v^{n_1 \lambda_h}}{u^{m_1\lambda_h}} \right). \eqno{(8)}$$
(Note that, on the left-hand side of (8), if we first intersect with $a_1$ or $b_1$, both of which are contained in some subset of the $\frak p_j, \frak q_l, \frak q'_k$, and $J_h$, followed by intersection with $v$ or $u$, the only elements that do not map to zero are the $[A/J_h]$.)



\begin{lemma}\label{threeterms} $ (u^{m_1}) \cap (v^{n_1}) \cap [A] - (v^{n_1}) \cap (u^{m_1}) \cap [A] =$

\hskip2in $ (u^{m_1}) \cap (a_1) \cap [A] - (a_1) \cap (u^{m_1}) \cap [A]$

\hskip2in $ + (b_1) \cap (v^{n_1}) \cap [A] - (v^{n_1}) \cap (b_1) \cap [A] $

\hskip2in $+ (a_1) \cap (v^{n_1}) \cap [A] - (b_1) \cap (u^{m_1}) \cap [A].$  
\end{lemma}

\begin{proof}  We will use the following fact, for $x,y \in A: (x) \cap (y) \cap [A] = (x) \cap \div(y/x)$.

$(u^{m_1}) \cap (v^{n_1}) \cap [A] - (v^{n_1}) \cap (u^{m_1}) \cap [A] $

$=(u^{m_1}) \cap \div(v^{n_1}/u^{m_1}) - (v^{n_1}) \cap \div(u^{m_1}/v^{n_1})$

$=(u^{m_1}) \cap \div(a_1/b_1) - (v^{n_1}) \cap \div(b_1/a_1) $

$=(u^{m_1}) \cap \div(a_1) - \underbrace{(u^{m_1}) \cap \div(b_1) + (v^{n_1}) \cap \div(a_1)}_{\text{equations }(7), (8)} - (v^{n_1}) \cap \div(b_1)$

$=(u^{m_1}) \cap \div(a_1) + (b_1) \cap \div(a_1) - (a_1) \cap \div(b_1) - (v^{n_1}) \cap \div(b_1)$

$=(u^{m_1}) \cap \div(a_1) + (b_1) \cap \div(a_1/b_1) - (a_1) \cap \div(b_1/a_1) - (v^{n_1}) \cap \div(b_1)$

$=(u^{m_1}) \cap \div(a_1) + (b_1) \cap \div(v^{n_1}/u^{m_1}) - (a_1) \cap \div(u^{m_1}/v^{n_1}) - (v^{n_1}) \cap \div(b_1)$

$=(u^{m_1}) \cap (a_1) \cap [A] - (a_1) \cap (u^{m_1}) \cap [A] +  (b_1) \cap (v^{n_1}) \cap [A] - (v^{n_1}) \cap (b_1) \cap [A] $

\hskip.5in $+ (a_1) \cap (v^{n_1}) \cap [A] - (b_1) \cap (u^{m_1}) \cap [A]$.

\end{proof}


Lemma \ref{threeterms} represents the difference $ (u^{m_1}) \cap (v^{n_1}) \cap [A] - 
(v^{n_1}) \cap (u^{m_1}) \cap [A]$ as a sum of three terms, each of which is
itself a difference of two terms.  We will establish our theorem by
computing each of these three differences and combining the results.

First, we need to find the remaining elements $a_i, b_i$ for $2 \leq i \leq r$.  Again we use the Approximation Theorem, and always we set $\displaystyle{b_i = \frac{u^{m_i}a_i}{v^{n_i}}}$ once we have chosen $a_i$.  The basic idea is that the element $a_i$ will always be chosen in the $\frak q_l$'s but never in the $\frak q_k'$'s, and the pair $a_i, b_i$ will never be contained in the same $\frak p_j$'s.
To be specific, we  choose $a_2$  
so that it is contained in $\frak p_{G+1}, \dots, \frak p_r$, with (positive) orders $n_2m_{G+1} - m_2n_{G+1}, \dots, n_2m_r - m_2n_r$, but is {\it not} contained in (1) $\frak p_1, \dots, \frak p_G$, (2) any of the $\frak q_k'$, or (3) any of the $J_h$.  In $\frak q_l$ it will have order $n_2t_l$.
We follow the same process for $a_3, \dots, a_G$, and note that none of $b_3, \dots, b_G$ is contained in any $\frak p_j$.  At the next step, the distribution of the $\frak p_j$ will change: we choose $a_{G+1}$ such that 
$$\div(a_{G+1}) = \sum_{j=G+2}^r (n_{G+1}m_j - m_{G+1}n_j) [A/\frak p_j] + \sum_{l=1}^f n_{G+1}t_l [A/\frak q_l] + \sum_h \mu_h [A/I_h],$$ where $n_{G+1}m_j - m_{G+1}n_j \geq 0$ and where the $I_h$ are a finite number of height one primes of $A$ different from all previous collections of height one primes.  Note that
$$\div(b_{G+1}) = \sum_{j=1}^G (m_{G+1}n_j - n_{G+1}m_j) [A/\frak p_j] + \sum_{k=1}^e m_{G+1}s_k [A/\frak q'_k] + \sum_h \mu_h [A/I_h],$$
where $m_{G+1}n_j - n_{G+1}m_j > 0$.  Note that $b_{G+1}$ is contained in $\frak p_1, \dots, \frak p_G$ and no other $\frak p_j$, while $a_{G+1}$ is contained in some subset of $\frak p_{G+2}, \dots, \frak p_r$.  From here we continue in this way to obtain all of the elements $a_i$ and $b_i$.  (Note that $a_r$ will not be contained in any of the $\frak p_j$.)  

It follows directly from the definitions of the $b_j$ that for every $j$ we have

\centerline{$\displaystyle{\frac{b_1^{n_j} a_j^{n_1}}{b_j^{n_1}} = \frac{a_1^{n_j}}{u^{\alpha_j}}}.$ }

The first term in Lemma \ref{threeterms} involves the pair $a_1, u^{m_1}$, which intersects on the primes $\frak p_{G+1}, \dots, \frak p_r$. 
Therefore, by the induction hypothesis, $$(u^{m_1}) \cap (a_1) \cap [A] - (a_1) \cap (u^{m_1}) \cap [A] = \sum_{j=G+1}^r m_1 \text{div} \left(\frak p_j, \frac{a_1^{n_j}}{u^{\alpha_j}} \right)$$
$$=\sum_{j=G+1}^r m_1 \text{div} \left(\frak p_j, \frac{b_1^{n_j}a_j^{n_1}}{b_j^{n_1}} \right)$$
$$= \sum_{j=G+1}^r m_1n_1 \text{div} \left(\frak p_j, \frac{a_j}{b_j} \right) + \sum_{j=G+1}^r m_1n_j  \text{div} \left(\frak p_j, b_1 \right).$$




To compute the second term in Lemma \ref{threeterms}, we note that 
the pair $v^{n_1}, b_1$ is a regular sequence, so
$(b_1) \cap (v^{n_1}) \cap [A] - (v^{n_1}) \cap (b_1) \cap [A] = 0.$



In the third term of Lemma \ref{threeterms}, we have
$$(a_1) \cap (v^{n_1}) \cap [A] = \sum_{j=1}^G m_jn_1 \div(\frak p_j, a_1),$$ since  $a_1\in \frak p_j$ for $j=G+1,\ldots, r$. We also have  
$$(b_1) \cap (u^{m_1}) \cap [A] = \sum_{j=1}^r m_1n_j \div(\frak p_j, b_1)$$
 since $b_1 \notin \frak p_j$ for any $j$.  Thus
$$(a_1) \cap (v^{n_1}) \cap [A] - (b_1) \cap (u^{m_1}) \cap [A]=
\sum_{j=1}^G m_jn_1 \div(\frak p_1, a_1)-\sum_{j=1}^r m_1n_j \div(\frak p_j, b_1).$$
Putting the three terms together, we have $(u^{m_1}) \cap (v^{n_1}) \cap [A] - (v^{n_1}) \cap (u^{m_1}) \cap [A] =$
$$\sum_{j=G+1}^r m_1n_1 \text{div} \left(\frak p_j, \frac{a_j}{b_j} \right) + \sum_{j=G+1}^r m_1n_j  \text{div} \left(\frak p_j, b_1 \right)$$
$$+ 0$$
$$ + \sum_{j=1}^G m_jn_1 \div(\frak p_j, a_1)-\sum_{j=1}^r m_1n_j \div(\frak p_j, b_1).$$
The first sum in this expression is in the form we want.  The remaining three sums combine
to give
$$\sum_{j=1}^G(m_jn_1 \div(\frak p_j, a_1)-m_1n_j \div(\frak p_j, b_1)). \eqno{(9)}$$
We recall that we have $n_1m_j = m_1n_j$ for  each $j=1, \dots, G$.  Consequently, the expression in (9) can be written as
$$=\sum_{j=1}^G(m_jn_1 \div(\frak p_j, a_1)-m_jn_1 \div(\frak p_j, b_1))$$
$$=\sum_{j=1}^Gm_jn_1 \div \left(\frak p_j, \frac{a_1}{b_1} \right).$$
In addition, it follows that
$$\left(\frac{v^{n_1}}{u^{m_1}}\right)^{m_j}=\left(\frac{v^{n_j}}{u^{m_j}}\right)^{m_1}$$
for each $j=1,\ldots, G$.  Since $a_j/b_j=v^{n_j}/u^{m_j}$ for each $j$, this implies that
$(a_1/b_1)^{m_j}=(a_j/b_j)^{m_1}$, so
$$m_j\div \left(\frak p_j,\frac{a_1}{b_1}\right)=m_1\div \left(\frak p_j,\frac{a_j}{b_j} \right).$$
Thus we have 
$$\sum_{j=1}^Gm_jn_1 \div \left(\frak p_j, \frac{a_1}{b_1} \right)=\sum_{j=1}^Gm_1n_1 \div \left(\frak p_j, \frac{a_j}{b_j} \right).$$
Putting this together with the first term finally gives
$$ (u^{m_1}) \cap (v^{n_1}) \cap [A] - 
(v^{n_1}) \cap (u^{m_1}) \cap [A] =\sum_{j=1}^r m_1n_1 \div \left(\frak p_j, \frac{a_j}{b_j} \right).$$
Dividing both sides of this equation by $m_1n_1$ now gives 
$$ (u) \cap (v) \cap [A] - 
(v) \cap (u) \cap [A] =\sum_{j=1}^r  \div \left(\frak p_j, \frac{a_j}{b_j} \right).$$
\end{proof}

We close with an example which demonstrates our choice of $a_i$ and $b_i$ and the cancelation that occurs.  In this instance, we have $r=3$.

\begin{example}  Let $A = k[x,w, \rho, y,z]$, where $k$ is a field.  Let $u = x^2w^3\rho z^2$ and $v = x^4w^6\rho^3y$ and set $\frak p_1 = (x), \frak p_2 = (w), \frak p_3 = (\rho), \frak q' = (z)$, and $\frak q = (y)$.
\vskip.25in

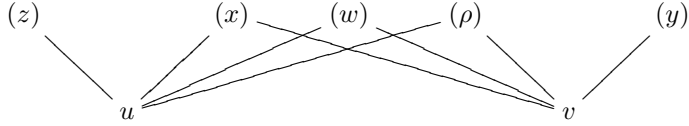
\xymatrix{
 & (z) \ar@{-}[dr] & & (x)  \ar@{-}[dl] \ar@{-}[drrr]  & (w) \ar@{-}[dll]  \ar@{-}[drr]  & (\rho) \ar@{-}[dlll]  \ar@{-}[dr]   & & (y) \ar@{-}[dl]  \\
    &  & u & & & &  v}
      

Using Definition \ref{map}, one can calculate that
 $$(u) \cap (v) \cap [A] = 2[A/(x,y)]+ 3[A/(y,w)] +[A/(\rho, y)] +2[A/(y,z)],$$

and $$(v) \cap (u) \cap [A] = 8[A/(x,z)] + 12[A/(w,z)] +6[A/(\rho, z)] +2[A/(y,z)].$$

Then, $(u) \cap (v) \cap [A] -  (v) \cap (u) \cap [A]$ 
$$ = \div \left((x), y^2/z^8 \right) + \div \left({(w)}, y^3/z^{12} \right) + \div \left({(\rho)}, y/z^6 \right). \eqno{(10)}$$ 
Using the ratios $v^2/u^4$, $v^3/u^6$, and $v/u^3$, choose $a_1 = \rho^2y^2, b_1 = z^8, a_2 = \rho^3 y^3, b_2 = z^{12}, a_3 = y$, and $b_3 = x^2w^3z^6.$  (In this case, no ideals $J_h$ come into play; i.e., the pair $a_1, b_1$ is a system of parameters.  Note that $\alpha_1 = \alpha_2 = 0$, and $\alpha_3 = 2$.)  One can check that the expression in equation (10) is equal to
$$= \div((x), a_1/b_1) + \div((w), a_2/b_2)+ \div \left({(\rho)}, a_3/b_3 \right).$$
\end{example}

\begin{ack}The authors would like to thank Mark Walker for pointing out the connection to $K$-theory
and Bill Fulton for bringing to our attention the paper by Andrew Kresch.  We would
also like to thank the referee for suggesting some simplifications in the
argument in section 5.
\end{ack}


\begin{thebibliography}{99}
\bibitem{F} W.~Fulton, {\em Intersection Theory}, Springer-Verlag, Berlin, 1984.
\bibitem{G} D.~Grayson, {\em Localization for flat modules in algebraic $K$-theory}, J. Algebra {\bf 61}, 463-496 (1979).
\bibitem{K} A.~Kresch, {\em Canonical rational equivalence of intersections of divisors}, Inv. Math {\bf 136} no. 3, 483-496 (1999).
\bibitem{M} H.~Matsumura, {\em Commutative Ring Theory}, Cambridge University Press, Cambridge, 1986.
\bibitem{Q} D.~Quillen, {\em Higher algebraic $K$-theory I}, Lecture Notes in Math. No. 341, Springer-Verlag, New York, 1973.
\bibitem{R} P.~Roberts, {\em Multiplicities and Chern Classes in Local Algebra}, Cambridge University Press, Cambridge, 1998. 
\end{thebibliography}
\end{document}